\documentclass[preprint,12pt]{elsarticle}
\pdfoutput=1

\usepackage{epsfig}
\usepackage{float}
\usepackage{graphicx}
\usepackage{times}
\usepackage{amsmath,amsthm}

\biboptions{numbers,sort&compress}

\usepackage{amssymb}

\journal{Applied Mathematics and Computation}

\begin{document}

\begin{frontmatter}

\title{Effects of interconnections among corruption, institutional punishment, and economic factors for the evolution of cooperation}

\author[label1,label2]{Linjie Liu}

\author[label2]{Xiaojie Chen\corref{cor}}
\cortext[cor]{Corresponding author} \ead{xiaojiechen@uestc.edu.cn}

\address[label1]{College of Science, Northwest A \& F University, Yangling 712100, China}
\address[label2]{ School of Mathematical Sciences, University of Electronic Science and Technology of China, Chengdu 611731, China}

\begin{abstract}
The view that altruistic punishment plays an important role in supporting public cooperation among human beings and other species has been widely accepted by the public. However, the positive role of altruistic punishment in enhancing cooperation will be undermined if corruption is considered. Recently, behavioral experiments have confirmed this finding and further investigated the effects of the leader's punitive power and the economic potential. Nevertheless, there are relatively few studies focusing on how these factors affect the evolution of cooperation from a theoretical perspective. Here, we combine institutional punishment public goods games with bribery games to investigate the effects of the above factors on the evolution of cooperation. Theoretical and numerical results reveal that the existence of corruption will reduce the level of cooperation when cooperators are more inclined to provide bribes. In addition, we demonstrate that stronger leader and richer economic potential are both important to enhance cooperation. In particular, when defectors are more inclined to provide bribes, stronger leaders can sustain the contributions of public goods from cooperators if the economic potential is weak.
\end{abstract}

\begin{keyword}

cooperation \sep evolutionary game theory  \sep institutional punishment \sep corruption

\end{keyword}

\end{frontmatter}

\newtheorem{thm}{\bf Theorem}[section]
\newtheorem{remark}{Remark}[section]

\section{Introduction}

Cooperation is the basis for the survival and reproduction of organisms, and its phenomenon is very common in the real society \citep{wangbmb16,Perc_PR_17,Perc_bio_10,wang2020epl}. However, how did cooperative behavior evolve remains a major puzzle in scientific community \citep{Kennedy_05}. Because individuals who choose cooperative behavior need to bear the cost by themselves but bring benefits to others, which makes them at a disadvantage in the process of natural selection \citep{Axelrod1984,han2015,Hofbauer1990,Tanimoto2017,WANG18AMC,Su2019,zhangis_21,Szolnoki2014interface,fu2010jtb,Szolnoki2022csf}. Previous researches have illuminated a number of mechanisms including reputation to explain the evolution of cooperation \citep{nowak2006,Kun2013,Tanimoto2018,Szolnoki2017pre,he2019,Szolnoki2017,Szolnoki2012,Szolnoki2012pre,Fupre2008,Szolnoki20amc,han18sr,Duong2021PRSA,xia2014amc,yang2017csf,yang2019pa,quan2021}. Recent studies on strong reciprocity theory reveal that some individuals will not hesitate to pay costs to punish those who do not cooperate, even if these costs cannot be compensated \citep{Vasconcelos2015,Okada2015,zhang2017,quan2019,Han2021,wang2019,chen2015pre,perc2012njp,perc2012sr,yang2020epl,zhang2021amc}. The threat of punishment can limit the prevalence of uncooperative individuals and thus support cooperation \cite{sigmund2010nature,chen2018pcb,quan2020epl}.

Although punishment plays a pivotal role in maintaining cooperation, it cannot be ignored that punishment, especially institutional punishment, is vulnerable to corruption. A number of studies support the finding that curruption might have a destructive effect on the role of punishment in promoting cooperation \citep{Abdallah_interface_14,lee_jtb_15,Lee_PNAS_19,Liu_MMMAS_19,Liu_IJBC_21}. Concretely, Abdallah \emph{et al.} \citep{Abdallah_interface_14} studied how corruption affects the co-evolution of cooperation and punishment in public goods games (PGG) and found that the effectiveness of institutional punishment on promoting cooperation will be undermined when bribery is an option. Subsequently, some studies based on donation game \citep{Lee_PNAS_19} and PGG \citep{Liu_MMMAS_19} revealed that cyclic behavior can be found when bribe between enforcers and free-riders happens.

Previous theoretical studies have investigated how corruption affects the effect of punishment on cooperation from different perspectives \citep{Abdallah_interface_14,lee_jtb_15,Lee_PNAS_19,huang_jtb_18,Liu_MMMAS_19,Liu_IJBC_21}. However, the impact of some important factors correlated with levels of corruption such as institutional and economic factors on the results has not been explored. Recent experimental research \citep{Muthukrishna_nhb_17} constructed the institutional punishment public goods game (IPGG) and bribery game (BG) where one individual is randomly selected from the game group to act as leader who could use taxes collected from all individuals for punishment. In addition, each individual can choose whether to bribe the leader or not, and the leader can choose to punish, accept bribes, or do nothing. The experimental results reveal that corruption, the punishment multiplier (the leader's punitive power), and the pool multiplier (the expression of the economic potential) play important roles in the evolution of cooperation. However, to our knowledge, thus far few theoretical works have revealed the causal interconnections among corruption, institutional punishment, and economic factors for the evolution of cooperation. Accordingly, it is still unclear the quantitative effects of these elements on the evolution of cooperation from a theoretical perspective.

In view of the above statements, in this work we construct a game-theoretical model by combing BG with IPGG to study the evolution of cooperation. We focus on the impacts of the pool multiplier and the punishment multiplier on the evolutionary results in an infinite populations. Through theoretical analysis, we reveal that corruption will decrease the level of cooperation when cooperators are more inclined to bribe, and stronger leaders and richer economic potential can both enhance the level of contributions. Once when defectors are more willing to provide bribe, a stronger leader can stimulate individual's willingness to cooperate under poor economic potential, while a richer economic environment requires a weaker leader. We further numerically verify the above theoretical results.

\section{Model}
\subsection{IPGG}
We consider that $N$ individuals are randomly sampled from an infinite well-mixed population to play the PGG. According to the previous experimental setup \citep{Muthukrishna_nhb_17}, we consider that every individual has an initial fund $b$ and decides whether or not to contribute $c$ to the public goods pool. And the total contribution of the public goods pool is multiplied by the pool multiplier $F$ where $1<F<N$ and then distributed equally to all group members regardless of contribution. Thus one cooperator ($C$) can obtain lower payoff than a defector ($D$) in the same group. In order to solve the above cooperation problem, we consider institutional punishment under the framework of PGG. Concretely, each individual needs to pay a fixed tax $\tau$ to support such a punishment institution before contributing to the PGG, and then one individual is randomly selected from the game group as a leader. With probability $\beta$, the leader chooses to punish individuals in the PGG by using taxes extracted from all individuals, with probability $1-\beta$ he/she does nothing. Here, we assume that one punishing leader will allocate the $\alpha$ ratio of total taxes to punish cooperators, while the remainder $(1-\alpha)N\tau$ is used for equally punishing the defectors. Here the punishment multiplier is set as $r_{p}$, which describes the intensity of the selected leader to execute the punishment. Besides, the pool multiplier $F$ is used to characterize the economic potential. In our model, the bigger the $r_{p}$ is, the stronger the punitive power of the leader is; The greater the $F$, the richer the economic potential. Thus the payoffs of $C$ and $D$ individuals in the IPGG can be respectively written as
\begin{eqnarray}\label{eq1}
\pi_{C}&=&b+\frac{Fc(N_{C}+1)}{N}-c-\tau-(1-\frac{1}{N})[\frac{N_{C}}{N-1}\frac{\beta\alpha N\tau r_{p}}{N_{C}}\nonumber\\
&+&\frac{N_{D}}{N-1}\frac{\beta\alpha N\tau r_{p}}{N_{C}+1}]
\end{eqnarray}
and
\begin{eqnarray}\label{eq2}
\pi_{D}&=&b+\frac{FcN_{C}}{N}-\tau-(1-\frac{1}{N})[\frac{N_{C}}{N-1}\frac{\beta(1-\alpha)N\tau r_{p}}{N_{D}+1}\nonumber\\
&+&\frac{N_{D}}{N-1}\frac{\beta(1-\alpha)N\tau r_{p}}{N_{D}}],
\end{eqnarray}
where $1-\frac{1}{N}$ denotes the probability that the focus cooperator or defector is not selected as the leader, $N_{C}$ and $N_{D}$ respectively refer to the number of $C$ and $D$ individuals in the group. Besides, $\frac{N_{C}}{N-1}$ denotes the probability that one cooperator is selected as leader among the remaining $N-1$ individuals, and $\frac{\beta\alpha N\tau r_{p}}{N_{C}}$ denotes the expected fine that the focus cooperator needs to bear when the leader is a cooperator. $\frac{N_{D}}{N-1}\frac{\beta\alpha N\tau r_{p}}{N_{C}+1}$ represents the expected fine that the focus cooperator receives from the defective leader. Similarly, $\frac{N_{C}}{N-1}\frac{\beta(1-\alpha)N\tau r_{p}}{N_{D}+1}+\frac{N_{D}}{N-1}\frac{\beta(1-\alpha)N\tau r_{p}}{N_{D}}$ represents the expected fine that the focus defector needs to bear.

\subsection{BG}

Then we consider bribery under the framework of IPGG. In this case, all $N-1$ individuals and one leader respectively have one additional choice. Concretely, the other $N-1$ individuals can decide to use their endowment to contribute to the common pool, retain for themselves, and offer bribes to the selected leader. In turn, the leader can choose to punish other individuals, do nothing, or accept the bribe. Here we set the probability that one cooperator (non-leader) chooses to offer bribe $h$ to the leader as $p$ and the probability that one defector chooses to offer bribe as $q$. Furthermore, the probability that the leader chooses to punish is set as $\beta$, the probability of accepting bribe is $\gamma$, and hence the probability of doing nothing is $1-\beta-\gamma$. Then the payoffs of $C$ and $D$ individuals obtained from the BG can be respectively rewritten as
\begin{eqnarray}\label{eq3}
\pi_{C}&=&b+\frac{Fc(N_{C}+1)}{N}-c-\tau+\frac{1}{N}(p\gamma hN_{C}+qhN_{D}\gamma)\nonumber\\
&-&(1-\frac{1}{N})[ph\gamma+\frac{N_{C}}{N-1}\frac{\beta\alpha N\tau r_{p}}{N_{C}}+\frac{N_{D}}{N-1}\frac{\beta\alpha N\tau r_{p}}{N_{C}+1}]
\end{eqnarray}
and
\begin{eqnarray}\label{eq4}
\pi_{D}&=&b+\frac{FcN_{C}}{N}-\tau+\frac{1}{N}(p\gamma hN_{C}+qhN_{D}\gamma)\nonumber\\
&-&(1-\frac{1}{N})[qh\gamma+\frac{N_{C}}{N-1}\frac{\beta(1-\alpha)N\tau r_{p}}{N_{D}+1}+\frac{N_{D}}{N-1}\frac{\beta(1-\alpha)N\tau r_{p}}{N_{D}}],
\end{eqnarray}
where $\frac{1}{N}$ denotes the probability that the focus individual is selected as leader, $p\gamma hN_{C}$ denotes the expected bribe amount that the leader receives from cooperators, and $qhN_{D}\gamma$ denotes the expected bribe amount that the leader receives from defectors. Besides, $ph\gamma$ and $qh\gamma$ respectively denote the expected bribe amounts provided by cooperator and defector for the leader.

We use social learning to describe the strategy selection process, that is, individuals tend to imitate other individuals' strategies which can produce higher payoffs \citep{Hofbauer_98}. In the following, we shall explore the replicator dynamics of
cooperation and defection in an infinite population.

\section{Results}

In an infinite population, a general approach to investigate the evolutionary
dynamics of the fraction $x$ of $C$ individuals (and $1-x$ of $D$ individuals) in an infinite population
is the gradient of selection associated with the well-known replicator equation \citep{Schuster_83,Hofbauer_98}, given as
\begin{eqnarray}\label{re}
\dot{x}=x(1-x)(f_{C}-f_{D}),
\end{eqnarray}
which describes the change in the frequency of cooperation in the population. Here $f_{C}$ and $f_{D}$ are respectively
the average payoffs of $C$ and $D$ individuals. According to the
system equation (\ref{re}), $C (D)$ individuals will increase in the population whenever $\dot{x}> 0
(\dot{x}<0)$. Since the population is well-mixed, where each individual
can interact with each other. Then we can written the average payoffs of $C$ and $D$ individuals as
\begin{eqnarray}\label{equ6}
   f_{C}&=&\sum_{N_{C}=0}^{N-1}\binom{N-1}{N_{C}}x^{N_{C}}(1-x)^{N-N_{C}-1}\pi_{C}
\end{eqnarray}
and
\begin{eqnarray}\label{equ7}
   f_{D}&=&\sum_{N_{C}=0}^{N-1}\binom{N-1}{N_{C}}x^{N_{C}}(1-x)^{N-N_{C}-1}\pi_{D}.
\end{eqnarray}

\begin{figure}[H]
\centering
\includegraphics[width=1\linewidth]{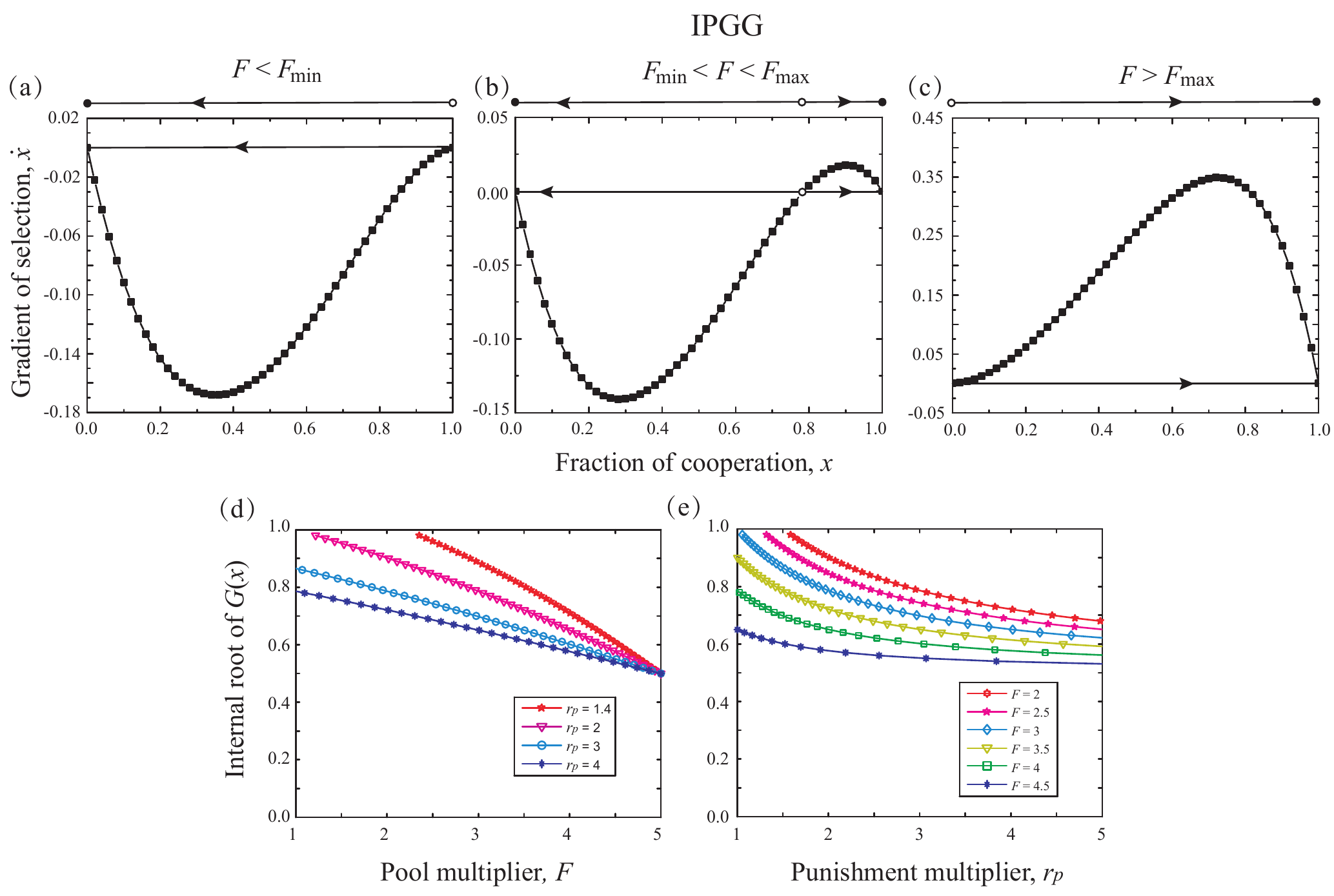}
\caption{Evolutionary dynamics of cooperation in IPGG. Panels (a), (b), and (c) show the gradient of selection changing with the fraction of cooperators for different values of pool multiplier $F$. Panel (d) shows the internal roots of $G(x)$ as a function of $F$ for different values of $r_{p}$. Panel (e) shows the internal roots of $G(x)$ as a function of $r_{p}$ for different values of $F$. Open circles represent unstable equilibrium points, and filled circles denote stable equilibrium points. The arrow pointing to the right indicates that cooperation is favored over defection. Parameter values: $b=12, N=5, c=1, \tau=1, F=2, \alpha=0.5, \beta=0.2,$ and $r_{p}=1.4$ in panel (a); $N=5, c=1, b=12, \tau=1, \alpha=0.5, F=3, r_{p}=2,$ and $\beta=0.2$ in panel (b); $N=5, F=4.7, c=1, b=12, \tau=1, \alpha=0.15, \beta=0.2,$ and $r_{p}=4$ in panel (c); $N=5, c=1, b=12, \tau=1, \alpha=0.5,$ and $\beta=0.2$ in panels (d) and (e).}
\label{fig1}
\end{figure}

For the convenience of theoretical analysis, we set that $G (x) = x(1-x)Q(x)$ where $Q (x) = f_{C}-f_{D}$.

\begin{remark}
The system equation (\ref{re}) has two boundary equilibrium points, namely, $x = 1$ (full cooperation) and $x = 0$ (full defection).
\end{remark}

In the following, we will analyze the existence and stability of the interior equilibrium points of equation (\ref{re}) in the models of IPGG and BG, respectively.

\subsection{Evolutionary dynamics in IPGG}

When bribery is not considered, by substituting equation (\ref{eq1}) into equation (\ref{equ6}), we can calculate the average payoff of cooperators as
\begin{eqnarray*}
   f_{C}&=&b+\frac{Fc[(N-1)x+1]}{N}-c-\tau-\beta\alpha\tau r_{p}-\beta\alpha\tau r_{p}\frac{(1-x)-(1-x)^{N}}{x}.
\end{eqnarray*}
Similarly, by substituting equation (\ref{eq2}) into equation (\ref{equ7}), we can calculate the average payoff of defectors as
\begin{eqnarray*}
   f_{D}&=&b+\frac{Fc(N-1)x}{N}-\tau-\beta(1-\alpha)\tau r_{p}-\beta(1-\alpha)\tau r_{p}\frac{x(1-x^{N-1})}{1-x}.
\end{eqnarray*}

\begin{thm}
Let $F_{\min}=\frac{[c-\beta\tau r_{p}+2\alpha\beta\tau r_{p}-\beta(1-\alpha)\tau r_{p}(N-1)]N}{c}$ and $F_{\max}=\frac{[c-\beta\tau r_{p}+\beta\alpha\tau r_{p}(N+1)]N}{c}$. Thus we have the following three conclusions:\\
(1) If $F< F_{\min},$ the system equation (\ref{re}) does not have interior equilibrium point for $x\in (0, 1)$. Then the equilibrium point $x=0$ is stable, while the equilibrium point $x=1$ is unstable.\\
(2) If $F_{\min}<F<F_{\max},$ the system equation (\ref{re}) just has one interior equilibrium point $x^{*}$, which is unstable. The two boundary equilibrium points $x=1$ and $x=0$ are stable.\\
(3) If $F> F_{\max},$ the system equation (\ref{re}) has no interior equilibrium point in $(0, 1)$. Only $x=1$ is a stable equilibrium point, while the equilibrium point $x=0$ is unstable .
\end{thm}

\begin{proof}
Considering that the difference between $f_{C}$ and $f_{D}$ determines the appearance of interior equilibrium points of replicator equation, we can theoretically analyze the existence of the interior equilibrium point. The difference between $f_{C}$ and $f_{D}$ can be calculated as
\begin{eqnarray}\label{eq8}
  Q(x) &=&f_{C}-f_{D}\nonumber\\
       &=&\frac{Fc}{N}-c-2\beta\alpha\tau r_{p}+\beta\tau r_{p}-\beta\alpha\tau r_{p}\frac{1-x}{x}[1-(1-x)^{N-1}]\nonumber\\
&+&\frac{r_{p}\beta(1-\alpha)\tau}{1-x}x(1-x^{N-1}).
\end{eqnarray}
Considering that $1-x^{N-1}=(1-x)(1+x+\cdots+x^{N-2})$, we can simplify equation (\ref{eq8}) to
\begin{eqnarray}\label{eq9}
  Q(x) &=&\frac{Fc}{N}-c-2\beta\alpha\tau r_{p}+\beta\tau r_{p}-\beta\alpha\tau r_{p}\sum_{k=0}^{N-2}(1-x)^{k+1}\nonumber\\
&+&r_{p}\beta(1-\alpha)\tau\sum_{k=0}^{N-2}x^{k+1}.
\end{eqnarray}
Notice that
\begin{eqnarray}\label{eq10}
  \frac{\partial Q(x)}{\partial x} &=&\beta\alpha\tau r_{p}\sum_{k=0}^{N-2}(k+1)(1-x)^{k}+r_{p}\beta(1-\alpha)\tau\sum_{k=0}^{N-2}(k+1)x^{k}.
\end{eqnarray}
We can judge that $Q(x)$ is an increasing function since $\frac{\partial Q(x)}{\partial x}>0$ for $x\in (0,1)$ when $\beta\tau r_{p}\neq 0$. In addition, we know that
\begin{eqnarray}\label{eq11}
Q(0)&=&\frac{Fc}{N}-c+\beta\tau r_{p}-\beta\alpha\tau r_{p}(N+1),\\
Q(1)&=&\frac{Fc}{N}-c+\beta\tau r_{p}-2\beta\alpha\tau r_{p}+\beta(1-\alpha)\tau r_{p}(N-1).
\end{eqnarray}
Then we can get the following conclusions:\\
(1) When $Q(1)<0$, that is, $F< F_{\min}$, the system equation (\ref{re}) has no interior equilibrium point in $(0, 1)$. Since $\frac{\partial G(x)}{\partial x}|_{x=1}=-Q(1)>0$ and $\frac{\partial G(x)}{\partial x}|_{x=0}=Q(0)<0$, we know that the equilibrium point $x=0$ is stable, while another boundary equilibrium point $x=1$ is unstable.\\
(2) When $Q(0)<0<Q(1)$, that is, $F_{\min}< F< F_{\max}$, the system equation (\ref{re}) has an unstable interior equilibrium point in $(0, 1)$. Since $\frac{\partial G(x)}{\partial x}|_{x=1}<0$ and $\frac{\partial G(x)}{\partial x}|_{x=0}<0$, thus the boundary equilibrium points $x=1$ and $x=0$ are both stable.\\
(3) When $Q(0)>0$, that is, $F> F_{\max}$, the system equation (\ref{re}) does not have interior equilibrium point for all $x\in (0, 1)$. Since $\frac{\partial G(x)}{\partial x}|_{x=1}<0$ and $\frac{\partial G(x)}{\partial x}|_{x=0}>0$, thus the equilibrium point $x=0$ is unstable, while another boundary equilibrium point $x=1$ is stable.
\end{proof}

In order to verify the above theoretical analysis, we provide numerical calculations by presenting three typical dynamical outcomes for the gradient of selection $\dot{x}$ changing with the proportion of cooperators, as shown in Fig.~\ref{fig1}. We find that when parameter values satisfy $F<F_{\min}$, the values of the gradient of selection $\dot{x}$ are negative, and thus cooperators will not survive in the population (Fig.~\ref{fig1} (a)). When $F_{\min}< F< F_{\max}$, as shown in Fig.~\ref{fig1} (b), one unstable interior equilibrium point can emerge, and thus the system can be divided into two basins of attraction. Concretely, the system evolves towards full cooperation ($x=1$) or full defection ($x=0$), depending on initial conditions. The existing two boundary equilibrium points, namely, $x = 0$ and $x = 1$ are both stable. When $F> F_{\max}$, the gradient of selection is always positive, and $x=1$ is the only steady state of the system (Fig.~\ref{fig1} (c)). In this case, cooperation can emerge and be maintained in the population. To better understand the effects of the pool multiplier $F$ and punishment multiplier $r_{p}$ on evolutionary dynamics, in Fig.~\ref{fig1}(d) and (e) we respectively show how the value of internal root changes with $F$ and $r_{p}$. We find that the value of the existing internal root monotonically decreases with increasing the pool multiplier $F$ (Fig.~\ref{fig1} (d)) and  punishment multiplier $r_{p}$ (Fig.~\ref{fig1} (e)). This means that the advantage of cooperation expands with the increase of any one of these two parameters ($F$ and $r_{p}$).

\subsection{Evolutionary dynamics in BG}
When bribery is possible, by substituting equation (\ref{eq3}) into equation (\ref{equ6}), we can calculate the average payoff of cooperators as
\begin{eqnarray*}
   f_{C}&=&b+\frac{Fc[(N-1)x+1]}{N}-c-\tau-\beta\alpha\tau r_{p}-\beta\alpha\tau r_{p}\frac{(1-x)-(1-x)^{N}}{x}\nonumber\\
&-&\frac{N-1}{N}p\gamma h+\frac{1}{N}[p\gamma h(N-1)x+q\gamma h (N-1)(1-x)].
\end{eqnarray*}
By substituting equation (\ref{eq4}) into equation (\ref{equ7}), we can calculate the average payoff of defectors as
\begin{eqnarray*}
   f_{D}&=&b+\frac{Fc[(N-1)x]}{N}-\tau-\beta(1-\alpha)\tau r_{p}-\beta(1-\alpha)\tau r_{p}\frac{x(1-x^{N-1})}{1-x}\\
&-&\frac{N-1}{N}q\gamma h+\frac{1}{N}[p\gamma h(N-1)x+q\gamma h (N-1)(1-x)].
\end{eqnarray*}

\begin{figure}[H]
\centering
\includegraphics[width=1\linewidth]{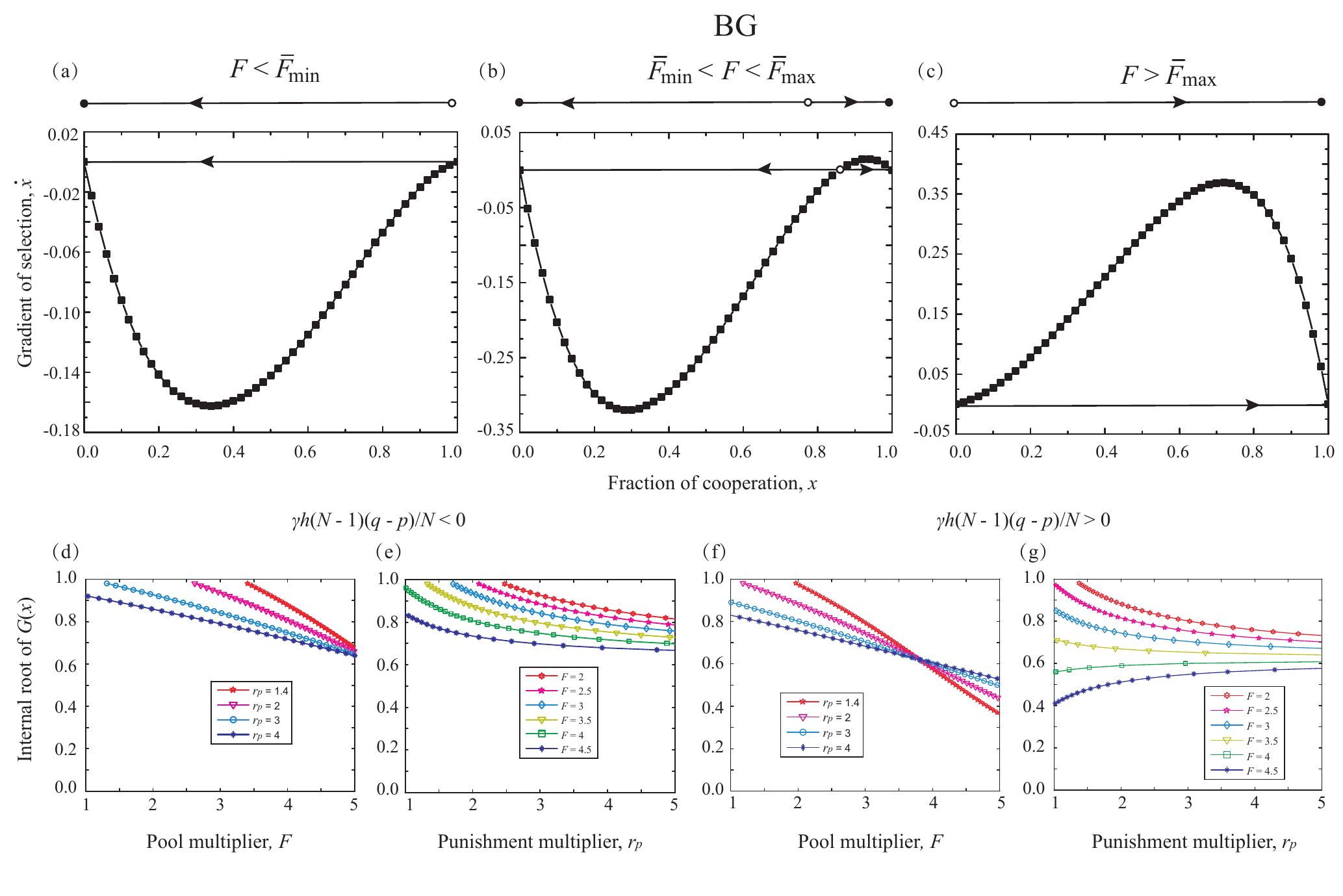}
\caption{Evolutionary dynamics of cooperation in BG. Panels (a)-(c) show the gradient of selection changing with the fraction of cooperators. Panels (d) and (e) respectively show the internal roots of $G(x)$ as functions of $F$ and $r_{p}$ when $\gamma h(N-1)(q-p)/N<0$. Panels (f) and (g) respectively show the internal roots of $G(x)$ as functions of $F$ and $r_{p}$ when $\gamma h(N-1)(q-p)/N>0$. Parameter values: $N=5, F=1.5, c=1, b=12, \tau=1, \alpha=0.6, h=1, \gamma=0.6, p=0.3, q=0.8, \beta=0.2,$ and $r_{p}=1.4$ in panel (a); $N=5, c=1, b=12, F=2, r_{p}=4, \tau=1, \alpha=0.6, h=1, \gamma=0.6, p=0.6, q=0.5,$ and $\beta=0.2$ in panel (b); $N=5, F=4, c=1, b=12, \tau=1, \alpha=0.15, \beta=0.2, h=1, \gamma=0.6, p=0.3, q=0.8,$ and $r_{p}=4$ in panel (c); $N=5, c=1, b=12, \tau=1, \alpha=0.6, \beta=0.2, h=1, \gamma=0.6, p=0.6,$ and $q=0.5$ in panels (d) and (e); $N=5, c=1, b=12, \tau=1, \alpha=0.6, \beta=0.2, h=1, \gamma=0.6, p=0.3,$ and $q=0.8$ in panels (f) and (g).}
\label{fig2}
\end{figure}

\begin{thm}
Let $\bar{F}_{\min}=\frac{Nc-(N-1)\gamma h(q-p)-N\beta\tau r_{p}(1-2\alpha)-N(N-1)r_{p}\beta(1-\alpha)\tau}{c}$ and $\bar{F}_{\max}=\frac{Nc-(N-1)\gamma h(q-p)-N\beta\tau r_{p}+N\beta\alpha\tau r_{p}(N+1)}{c}$. Thus we have the following three conclusions:\\
(1) If $F< \bar{F}_{\min},$ the system equation (\ref{re}) does not have interior equilibrium point for $x\in (0, 1)$. Then the boundary equilibrium point $x=0$ is stable, while another boundary equilibrium point $x=1$ is unstable.\\
(2) If $\bar{F}_{\min}<F<\bar{F}_{\max},$ the system equation (\ref{re}) has one interior equilibrium point $x^{*}$, which is unstable. The existing two boundary equilibrium points ($x=1$ and $x=0$) are stable.\\
(3) If $F> \bar{F}_{\max},$ the system equation (\ref{re}) has no interior equilibrium point in $(0, 1)$. Only the equilibrium point $x=1$ is stable, while the equilibrium point $x=0$ is unstable.
\end{thm}

\begin{proof}
In this case, the difference between $f_{C}$ and $f_{D}$ can be calculated as
\begin{eqnarray}\label{eq13}
Q(x)&=&\frac{Fc}{N}-c+\frac{N-1}{N}\gamma h(q-p)-2\beta\alpha\tau r_{p}+\beta\tau r_{p}-\beta\alpha\tau r_{p}\sum_{k=0}^{N-2}(1-x)^{k+1}\nonumber\\
&+&r_{p}\beta(1-\alpha)\tau\sum_{k=0}^{N-2}x^{k+1}.
\end{eqnarray}
Considering that $\frac{\partial Q(x)}{\partial x}>0$, we know that $Q(x)$ is strictly increasing for all $x\in (0,1)$. Accordingly, we know
\begin{eqnarray*}\label{eq11}
\frac{\partial G(x)}{\partial x}|_{x=0}&=&Q(0)\nonumber\\
&=&\frac{Fc}{N}-c+\frac{N-1}{N}\gamma h(q-p)+\beta\tau r_{p}-\beta\alpha\tau r_{p}(N+1),\\
-\frac{\partial G(x)}{\partial x}|_{x=1}&=&Q(1)\nonumber\\
&=&\frac{Fc}{N}-c+\frac{N-1}{N}\gamma h(q-p)+\beta\tau r_{p}-2\beta\alpha\tau r_{p}\nonumber\\
&+&\beta(1-\alpha)\tau r_{p}(N-1).
\end{eqnarray*}
We can get the following three conclusions:\\
(1) If $Q(1)<0$, that is, $F< \bar{F}_{\min}$, the system equation (\ref{re}) has no interior equilibrium point for $x\in (0, 1)$. Since $\frac{\partial G(x)}{\partial x}|_{x=1}>0$ and $\frac{\partial G(x)}{\partial x}|_{x=0}<0$, thus we know that $x=0$ is stable, while $x=1$ is unstable.\\
(2) If $Q(0)<0<Q(1)$, that is, $\bar{F}_{\min}< F< \bar{F}_{\max}$, the system equation (\ref{re}) has one interior equilibrium point $x=x^{*}$. Considering that $\frac{\partial G(x)}{\partial x}|_{x=x^{*}}>0$, we know that $x=x^{*}$ is unstable. Since $\frac{\partial G(x)}{\partial x}|_{x=1}<0$ and $\frac{\partial G(x)}{\partial x}|_{x=0}<0$, we know that the existing two boundary equilibrium points are stable.\\
(3) If $Q(0)>0$, that is, $F>\bar{F}_{\max}$, the system equation (\ref{re}) does not have interior equilibrium point for $x\in (0, 1)$. Since $\frac{\partial G(x)}{\partial x}|_{x=1}<0$ and $\frac{\partial G(x)}{\partial x}|_{x=0}>0$, we know that $x=1$ is stable, while $x=0$ is unstable.
\end{proof}

\begin{figure}[H]
\centering
\includegraphics[width=1\linewidth]{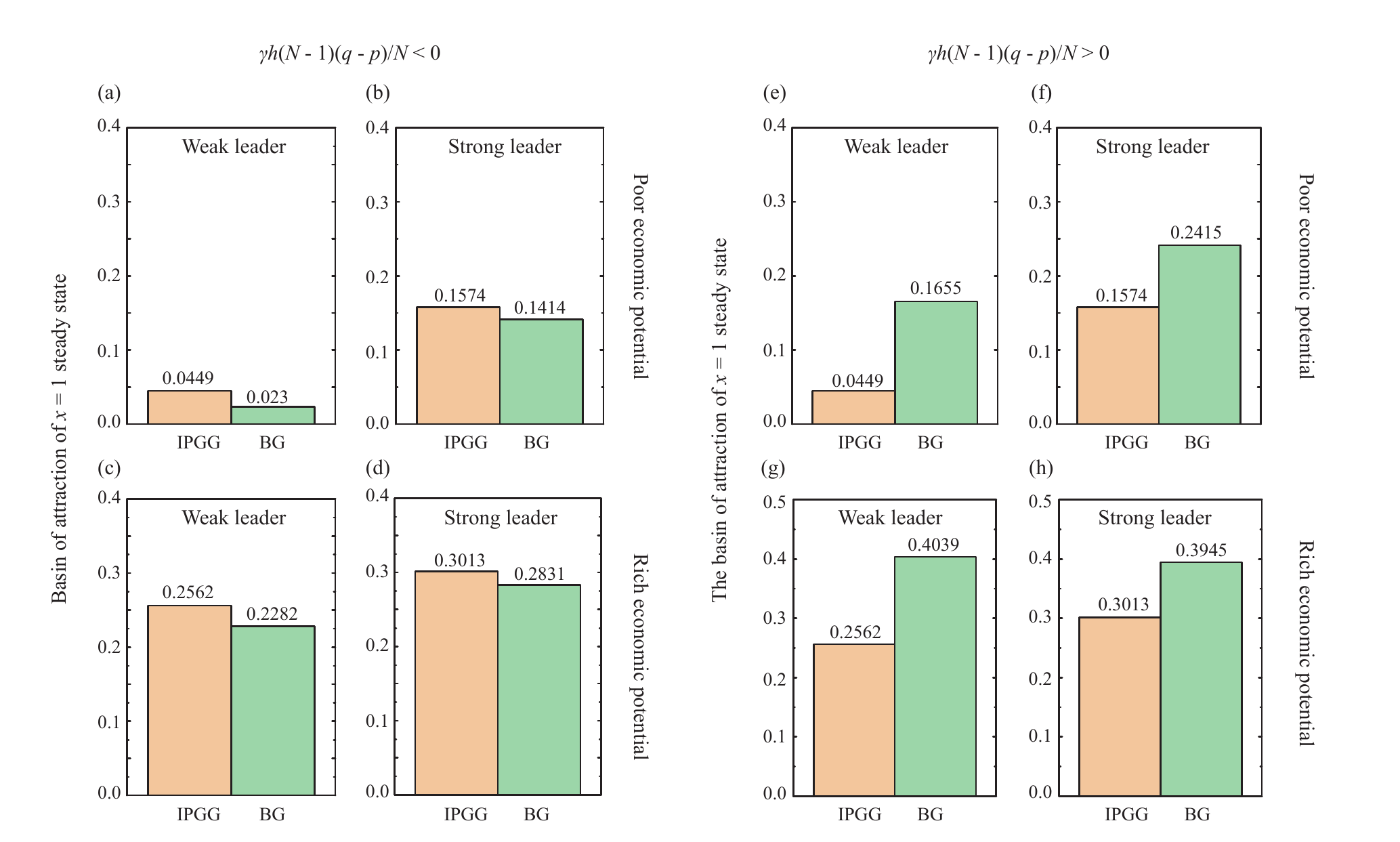}
\caption{Evolutionary dynamics of cooperation for different values of pool multiplier $F$ and punishment multiplier $r_{p}$ in IPGG and BG under two different scenarios.
Panels (a)-(d) show the basin of attraction of full cooperation state for different values of $F$ and $r_{p}$ when $\gamma h(N-1)(q-p)/N<0$. Panels (e)-(h) show the evolutionary outcomes when $\gamma h(N-1)(q-p)/N>0$. Parameter values: $N=5, c=1, b=12, \tau=1, p=0.6, q=0.5, \alpha=0.6, h=1, \gamma=0.6, F=2, \beta=0.2$, $r_{p}=2.5$ in panel (a), $r_{p}=4$ in panel (b); $N=5, c=1, b=12, \tau=1, \alpha=0.6, h=1, \gamma=0.6, p=0.6, q=0.5, F=4, \beta=0.2$, $r_{p}=2.5$ in panel (c), $r_{p}=4$ in panel (d); $N=5, c=1, b=12, \tau=1, \alpha=0.6, h=1, \gamma=0.6, p=0.3, q=0.8, F=2, \beta=0.2$, $r_{p}=2.5$ in panel (e), $r_{p}=4$ in panel (f); $N=5, c=1, b=12, \tau=1, \alpha=0.6, h=1, \gamma=0.6, p=0.3, q=0.8, F=4, \beta=0.2$, $r_{p}=2.5$ in panel (g), $r_{p}=4$ in panel (h).}
\label{fig3}
\end{figure}

In Fig.~\ref{fig2}, we provide three representative dynamical outcomes, namely, defection dominance (Fig.~\ref{fig2} (a)), bistable state of cooperation and defection (Fig.~\ref{fig2} (b)), and cooperation dominance (Fig.~\ref{fig2} (c)) when the model parameter value respectively satisfy $F<\bar{F}_{\min}, \bar{F}_{\min}<F<\bar{F}_{\max}$, and $F>\bar{F}_{\max}$. Besides, we show how the internal root of the gradient of selection varies with the model parameters $r_{p}$ and $F$ when $\gamma h(N-1)(q-p)/N<0$. As shown in Fig.~\ref{fig2}(d) and (e), for fixed $r_{p}$, the value of the existing internal root of $G(x)$ gradually decreases with the increase of $F$, which means that increasing $F$ can expand the basin of attraction of full cooperation state ($x=1$);
for fixed $F$, the increase of $r_{p}$ will enlarge the advantage of cooperators. Furthermore, Fig.~\ref{fig2}(f) and (g) show the evolutionary outcomes when $\gamma h(N-1)(q-p)/N>0$. Concretely, for fixed $r_{p}$, increasing $F$ will expand the basin of attraction of full cooperation steady state.
In addition, we notice that the larger the value of $r_{p}$, the slower the value of the internal root $x^{*}$ decreases with the increase of $F$ (Fig.~\ref{fig2}(f)). Fig.~\ref{fig2}(g) shows that with the increase of $r_{p}$ the internal root of $G(x)$ monotonically decreases when $F$ is not large. While when the value of $F$ is large, the value of the interior equilibrium point will increase with the increase of $r_{p}$.

In what follows, we respectively show how the basin of attraction of the full cooperation steady state varies for two different $r_{p}$ values when the economic potential is poor or rich under two different scenarios (Fig.~\ref{fig3}). Concretely, when $\gamma h(N-1)(q-p)/N<0$, we show that corruption will reduce the advantage of cooperation, and a stronger leader and stronger economic potential will both help to enhance the level of cooperation (see Fig.~\ref{fig3}(a)-(d)). While if $\gamma h(N-1)(q-p)/N>0$, a rich economic potential environment is more conducive to the evolution of cooperation than a poor economic potential environment. More interestingly, when the economic potential is poor, stronger leaders can be more favorable to expand the basin of attraction of full cooperation than a weak leader (see Fig.~\ref{fig3}(e) and (f)). While when the economic potential is strong, weaker leaders can be more beneficial to expand the advantage of cooperation in BG (see Fig.~\ref{fig3}(g) and (h)).

Finally, it is worth pointing out that we just set $N = 5$ in
Figs.\ref{fig1}-\ref{fig3} following previous related works \cite{huang_jtb_18,Liu_MMMAS_19,Liu_IJBC_21}, so that it will
be more conducive to make the comparison between our results
and those in Refs. \cite{huang_jtb_18,Liu_MMMAS_19,Liu_IJBC_21} where $N$ is also fixed at 5. Indeed
the group size $N$ plays an important role in evolution of cooperation \cite{Chen_15sr,Carpenter_07GEB}. However, we find that our main
results remain valid if the value of $N$ is approximately changed.

\section{Conclusions}

In summary, we have constructed a mathematical model by introducing bribery into the IPGG in an infinite population, so that we can study the interconnections among corruption, economic potential, and leader's punitive power and their effects on the replicator dynamics of cooperation and defection. We have found that there exist three different evolutionary outcomes, namely, cooperation dominance, bistable state of cooperation and defection, and defection dominance, and further given the conditions for the emergence of the above three dynamic behaviors from a theoretical perspective in both IPGG and BG. Furthermore, we have found that when cooperators are more inclined to offer bribes than defectors, the existence of corruption will reduce the advantage of cooperation, and richer economic potential and stronger leaders are both better able to sustain public goods contributions. On the contrary, when defectors are more willing to offer bribes, some nontrivial results will appear: stronger leaders in poor economic potential environment increase the advantage of cooperation, while reduce it in a rich economic potential environment.

Previous behavioral experimental results show that the leader's punitive power and the economic potential are two important factors that may influence the effect of institutional punishment on public good provisioning \citep{Muthukrishna_nhb_17}. Concretely, Muthukrishna \emph{et al.} \citep{Muthukrishna_nhb_17} revealed that stronger leaders can better promote the evolution of cooperation when the economic potential is weak and bribery is not considered, but when bribery is allowed, strong leaders will reduce the level of contribution. Furthermore, the existing of corruption will lead to the reduction of public goods provisioning. Different from previous experimental observations, our theoretical model predicts that whether corruption reduces the public goods contribution depends on the bribery tendency of cooperators and defectors, and whether strong leaders can promote cooperation depends on the economic potential. Our work may unveil more quantitative effects of interconnections among corruption, institutional punishment, and economic factors for the evolution of cooperation from theoretical perspective.

In the future work, we can thus consider different types of institutional incentive strategies such as institutional reward and the combination of punishment and reward. These incentives have been proved to be effective in solving social dilemmas \cite{Sun_21iscience,Gois19sr} and regulating technology safety \cite{HAN22TS,Pereira20JAIR}. Thus it is worth studying the interplay among different institutional incentives, corruption, and the evolution of cooperation, and further revealing how these incentives might influence bribery decision making. Furthermore, it could be interesting to explore whether these incentive strategies can promote the evolution of cooperation better than institutional punishment.

\section*{Acknowledgments}
This research was supported by the National Natural Science Foundation of China (Grants Nos. 61976048 and 62036002) and the Fundamental Research Funds of the Central Universities of China. L.L. acknowledges the support from Special Project of Scientific and Technological Innovation (Grant No. 2452021004).

\section*{Declaration of interest}
None

\end{document}